\documentclass[11pt,a4paper]{article}

\usepackage{amsmath,amsfonts,amssymb,amsthm}
\usepackage{bbm,mathrsfs, bm, mathtools}

\title{Minkowski type theorems for convex sets in cones}

\author{Rolf Schneider}
\date{}
\newcommand{\Sd}{{\mathbb S}^{d-1}}
\newcommand{\R}{{\mathbb R}}

\newcommand{\Rd}{{\mathbb R}^d}
\newcommand{\N}{{\mathbb N}}

\newcommand{\Ha}{\mathcal{H}}

\newcommand{\B}{\mathcal{B}}
\newcommand{\D}{{\rm d}}

\newtheorem{theorem}{Theorem}
\newtheorem{lemma}{Lemma}

\begin{document}
\maketitle

\begin{abstract}
Minkowski's classical existence theorem provides necessary and sufficient conditions for a Borel measure on the unit sphere of Euclidean space to be the surface area measure of a convex body. The solution is unique up to a translation. We deal with corresponding questions for unbounded convex sets, whose behavior at infinity is determined by a given closed convex cone. We provide an existence theorem and a stability result.
\\[1mm]
\end{abstract}

\section{Introduction}\label{sec1}

Minkowski's existence theorem is one of the classical results of convex geometry. It provides necessary and sufficient conditions for a Borel measure on the unit sphere of Euclidean space to be the surface area measure of a convex body. More precisely, let $\Sd$ be the unit sphere of Euclidean space $\R^d$ (with scalar product $\langle\cdot\,,\cdot\rangle$, norm $\|\cdot\|$, and origin $o$), and let $\B(\Sd)$ be its $\sigma$-algebra of Borel sets. Let $K\subset\R^d$ ($d\ge 2$) be a convex body (a nonempty compact convex set). For a Borel set $\omega\in\B(\Sd)$, let $\tau(K,\omega)$ be the reverse spherical image of $K$ at $\omega$, that is, the set of all boundary points of $K$ at which there exists an outer normal vector falling in $\omega$. Then $S_{d-1}(K,\omega):= \Ha^{d-1}(\tau(K,\omega))$, where $\Ha^{d-1}$ denotes the $(d-1)$-dimensional Hausdorff measure, defines the surface area measure $S_{d-1}(K,\cdot)$ of the convex body $K$. Minkowski's theorem says that a Borel measure $\varphi$ on $\Sd$ is the surface area measure of some convex body $K$ if and only if $\varphi$ is finite, not concentrated on a great subsphere, and satisfies $\int_{\Sd} u\,\varphi(\D u)=o$. Moreover, $K$ is uniquely determined up to a translation. For historical references, reproductions of proofs, and generalizations, we refer to \cite{Sch14}, Sections 8.2 and 9.2.

The surface area measure makes sense for unbounded convex sets, too. Let $K\subset \R^d$ be a nonempty closed convex set, which is not necessarily bounded. Its spherical image $S_K$ is defined as the set of all unit vectors that are outer normal vectors of $K$ at boundary points. Then $S_K$ is a spherically convex set, contained in a closed halfsphere if $K$ is unbounded. For Borel sets $\omega\in\B(S_K)$, the reverse spherical image $\tau(K,\omega)$ and the surface area measure $S_{d-1}(K,\omega)$ at $\omega$ can be defined as above. The total measure $S_{d-1}(K,S_K)$ is no longer finite if $K$ is unbounded.

In this paper, we are interested in Minkowski type theorems for unbounded convex sets, which are contained in a given convex cone and are such that their behavior at infinity is determined by the cone. Such a question can be viewed as a Minkowski problem with a boundary condition. We recall that the support function of a nonempty closed convex set $K\subset\R^d$ is defined by
$$ h(K,x):= \sup\{\langle x,y\rangle:y\in K\}\quad\mbox{for }x\in\R^d.$$
For $u$ in the relative interior of $S_K$, we have $h(K,u) =\max\{\langle x,y\rangle:y\in K\}$, which is finite. For $u$ in the relative boundary of $S_K$, it can happen that $h(K,u)$ is finite and attained, finite but not attained, or infinite. The question considered below can be viewed as a Minkowski problem with boundary condition, where a Borel measure is given on a spherically convex open set $\Omega_C\subset\Sd$ with closure in an open halfsphere, and the boundary condition requires that the support function of the solution set be zero on the boundary of $\Omega_C$. We remark that Minkowski type theorems  for unbounded polyhedra, with certain boundary conditions, are found in Alexandrov \cite[Sections 7.3, 7.4]{Ale05}. Again different versions of Minkowski problems with boundary conditions were treated, for example, by Busemann \cite{Bus59} and Oliker \cite{Oli82}.

In the following, $C\subset\R^d$ is a pointed closed convex cone (with apex at the origin) with interior points; it will be kept fixed. Let $K\subset C$ be a closed convex set. We distinguish several types of such sets. We say that $K$ is {\em $C$-asymptotic} if for $x\in{\rm bd}\,C$ the distance of $x$ from $K$ tends to zero as $\|x\|\to\infty$. If $C\setminus K$ has finite volume (Lebesgue measure), we say that $K$ is {\em $C$-close}. And finally the set $K$ is called {\em $C$-full} if $C\setminus K$ is bounded. Clearly, a $C$-full set is $C$-close, and a $C$-close set is $C$-asymptotic.

Convex sets of this kind have been considered in the literature. $C$-asymptotic convex sets have appeared, under different aspects, in the work of Gigena \cite{Gig78, Gig81}. Khovanski\u{\i} and Timorin \cite{KT14} introduced coconvex sets (essentially) as the differences $C\setminus K$, where $K$ is a $C$-full set. They were motivated by applications to algebraic geometry and singularity theory, but became interested in carrying over notions from convex geometry to coconvex sets. In particular, they introduced mixed volumes and derived Alexandrov--Fenchel inequalities for coconvex sets. Conconvex sets of finite volume were treated in \cite{Sch18}, as differences $C\setminus K$, where $K$ is a $C$-close set. In \cite{Sch18}, first Minkowski type theorems for convex sets in cones were obtained. 

To explain them, we note that the natural domain of definition for the surface area measure of a $C$-asymptotic convex set $K$ is the open subset
$$\Omega_C:= \Sd\cap{\rm int}\,C^\circ$$ 
of the unit sphere, where $C^\circ:=\{x\in \Rd:\langle x,y\rangle\le 0 \;\forall\, y\in C\}$ is the polar cone of $C$. Thus, in the following the surface area measure of $K$ is defined by $S_{d-1}(K,\omega)=\Ha^{d-1}(\tau(K,\omega))$ for Borel sets $\omega\in\B(\Omega_C)$ only. This defines a Borel measure $S_{d-1}(K,\cdot)$ on $\Omega_C$, which is not necessarily finite. The following existence theorem was proved in \cite{Sch18}. (The assumption, made there for notational reasons, that $\varphi$ be nonzero, can evidently be deleted.)

\vspace{2mm}

\noindent{\bf Theorem A.} {\em Let $\varphi$ be a finite Borel measure on $\Omega_C$ with compact support in $\Omega_C$. Then there exists a $C$-full set $K$ such that $S_{d-1}(K,\cdot)=\varphi$.}

\vspace{2mm}

We observe that, in contrast to Minkowski's classical theorem for convex bodies, the measure $\varphi$ need not have to satisfy any further condition. On the other hand, the condition of compact support is only sufficient, but not necessary: there are $C$-full sets $K$ for which the support of $S_{d-1}(K,\cdot)$ is the closure of $\Omega_C$ and thus not a compact subset of $\Omega_C$. 

Uniqueness holds, more generally, for $C$-close sets. Also the following uniqueness theorem was proved in \cite{Sch18}, by carrying over to coconvex sets of finite volume some arguments from the classical Brunn--Minkowski theory of convex bodies. 

\vspace{2mm}

\noindent{\bf Theorem B.} {\em Let $K,L$ be $C$-close sets such that $S_{d-1}(K,\cdot)=S_{d-1}(L,\cdot)$. Then $K=L$.}

\vspace{2mm}

Our first aim in this note is to prove an existence theorem in the style of Theorem A, where it is only assumed that $\varphi$ be finite.

\begin{theorem}\label{T1.1}
Let $\varphi$ be a finite Borel measure on $\Omega_C$. Then there exists a $C$-close set $K$ such that $S_{d-1}(K,\cdot)=\varphi$.
\end{theorem}

Although the measure $\varphi$ in this theorem is finite, we cannot assert that the body $K$ with $S_{d-1}(K,\cdot)=\varphi$ is $C$-full. In fact, if $d\ge 3$, one can construct $C$-close sets with finite surface area measure for which $C\setminus K$ is unbounded. An example is given in Section \ref{sec4}.

We remark that there are several unsolved problems. Necessary and sufficient conditions for surface area measures are unknown for  $C$-full convex sets as well as for $C$-close or $C$-asymptotic sets. It is also unknown (as already mentioned  in \cite{Sch18}) whether Theorem B can be extended to $C$-asymptotic sets.

Theorem B implies, in particular, the uniqueness of the $C$-full set whose existence is guaranteed by Theorem A. The second aim of this note is to improve this uniqueness result by a stability assertion. To formulate it, we need a metric for $C$-full sets and a metric for finite measures on $\Omega_C$. We can define the Hausdorff distance of $C$-full sets $K,L$ in the same way as for compact sets, namely by
$$ d_H(K,L) = \max\{\sup_{x\in K} \inf_{y\in L} \|x-y\|, \sup_{x\in L} \inf_{y\in K} \|x-y\|\}.$$
For a set $A\subseteq \Omega_C$ and for $\varepsilon>0$, let
$$ A_\varepsilon:=\{y\in\Omega_C:\|x-y\|<\varepsilon\mbox{ for some }x\in A\}.$$
For two finite Borel measures $\mu,\nu$ on $\Omega_C$, their L\'evy--Prokhorov distance is defined by
$$ \delta_{LP}(\mu,\nu):= \inf\{\varepsilon>0: \mu(A)\le \nu(A_\epsilon)+\varepsilon,\, \nu(A)\le \mu(A_\epsilon)+\varepsilon\;\forall\,A\in\B(\Omega_C)\}.$$
This defines a metric, which metrizes the weak convergence of finite measures on $\B(\Omega_C)$ (cf. \cite[Thm. 6.8]{Bil99}). The L\'evy--Prokhorov distance of surface area measures of convex bodies plays an important role in a characterization of the Blaschke addition by Gardner, Parapatits and Schuster \cite{GPS14}. It appears also in a stability theorem for convex bodies by Hug and Schneider \cite[Thm. 3.1]{HS02}, after which the following theorem is modelled.

\begin{theorem}\label{T1.2} 
Let $\omega$ be a compact subset of $\Omega_C$, and let $K,L$ be $C$-full sets whose surface area measures are concentrated on $\omega$.  There is a constant $c$, depending only on $C,\omega$ and an upper bound for $S_{d-1}(K,\cdot), S_{d-1}(L,\cdot)$, such that
$$ d_H(K,L)\le c\delta_{LP}(S_{d-1}(K,\cdot),S_{d-1}(L,\cdot))^{1/d}.$$
\end{theorem}

After some preliminaries in Section \ref{sec2}, Theorem \ref{T1.1} is proved in Section \ref{sec3} and Theorem \ref{T1.2} in Section \ref{sec5}. Section \ref{sec4} contains examples and a necessary condition.

\section{Preliminaries}\label{sec2}

By assumption, the fixed cone $C$ is pointed, hence its polar cone $C^\circ$ has nonempty interior. Therefore, we can choose a unit vector $-w\in{\rm int}\,C^\circ$, and then $\langle w,x\rangle >0$ for all $x\in C\setminus\{o\}$. We fix this vector $w$ in the following and write
$$ H(w,t):= \{x\in \R^d:\langle w,x\rangle =t\}, \qquad H^-(w,t):= \{x\in \R^d:\langle w,x\rangle \le t\}$$for $t\in\R$. Similar notation is used for other unit vectors. We abbreviate
$$ C_t:= C\cap H^-(w,t) \quad\mbox{for } t>0.$$
Thus, the sets $C_t$ are bounded. The Hausdorff distance of $C$-full sets defined above satisfies
$$ d_H(K,L) = d_H(K\cap C_t,L\cap C_t)$$
for all sufficiently large $t$, where on the right side we have the familiar Hausdorff distance of convex bodies.

Let $K$ be a $C$-asymptotic set. We have already defined the surface area measure and the support function of $K$.  The surface area measure is defined on $\Omega_C$, and we restrict the support function to the closure of $\Omega_C$. Then it is finite, and $h(K,\cdot)\equiv 0$ if and only if $K=C$. 

In the special case of a $C$-full set $K$, the volume (Lebesgue measure) of the coconvex set $C\setminus K$ is given by
\begin{equation}\label{2.2} 
V_d(C\setminus K) = -\frac{1}{d} \int_{\Omega_C} h(K,u)\,S_{d-1}(K,\D u).
\end{equation}
This is Lemma 1 in \cite{Sch18}.

We shall repeatedly need the following estimates.

\begin{lemma}\label{L2.0}
Let $K$ be a $C$-asymptotic set. If the largest ball $B$ with center $o$ and $B\cap{\rm int}\,K=\emptyset$ has radius $r$, then
\begin{equation}\label{2.a} 
-r\le h(K,u) \le 0\quad\mbox{for }u\in\Omega_C.
\end{equation}

If $S_{d-1}(K,\cdot)\le b<\infty$, then
\begin{equation}\label{2.b} 
r\le c_1
\end{equation}
with a constant $c_1$ depending only on $C$ and $b$.

Suppose that $\omega$ is a compact subset of $\Omega_C$. Then there is a number $a>0$, depending only on $C$ and $\omega$, such that
\begin{equation}\label{2.c}
a\|x\|\le|\langle x,u\rangle|\quad\mbox{for }x\in C \mbox{ and }u\in\omega.
\end{equation}
\end{lemma}

\begin{proof}
Let $K$ and $B$ be as in the lemma. Let $H$ be a supporting hyperplane of $K$, and suppose that $H\cap B=\emptyset$. Since ${\rm int}\,K$ lies entirely in the open halfspace bounded by $H$ that does not contain $o$, the ball $B$ can be increased without intersecting ${\rm int}\,K$, a contradiction. Thus, any supporting hyperplane of $K$ intersects $B$, which means that (\ref{2.a}) holds.

Suppose that $S_{d-1}(K,\cdot)\le b<\infty$. The mapping from ${\rm bd}\,K\cap{\rm int}\,C$ to ${\rm bd}\,B$ that is defined by $x\mapsto rx/\|x\|$ has Lipschitz constant 1. Its image is all of ${\rm bd}\,B\cap{\rm int}\,C$, since $K$ is a $C$-asymptotic set. It follows that
\begin{eqnarray*}
&& r^{d-1}\Ha^{d-1}(\Sd\cap {\rm int}\,C) = \Ha^{d-1}({\rm bd }\,B\cap{\rm int}\,C)\\
&& \le  \Ha^{d-1}({\rm bd}\,K\cap{\rm int}\,C)= S_{d-1}(K,\Omega_C) \le b,
\end{eqnarray*}
from which (\ref{2.b}) follows.

Now suppose that $\omega$ is a compact subset of $\Omega_C$. We note that $\omega$ has a positive distance from the boundary of $\Omega_C$. Hence, there is a number $a>0$ such that $\langle x,u\rangle \le -a$ for all $x\in C\cap\Sd$ and all $u\in\omega$. This yields (\ref{2.c}).
\end{proof}

Although the surface area measure of a $C$-asymptotic set $K$ is in general infinite, it is finite on compact sets. In fact, let $\omega$ be a compact subset of $\Omega_C$. For $u\in\omega$ and $x\in \tau(K,u)$, it follows from (\ref{2.c}) and (\ref{2.a}) that
$$ a\|x\|\le|\langle x,u\rangle|=|h(K,u)|\le r.$$
This shows that $\tau(K,\omega)$ is bounded and, hence, that $S_{d-1}(K,\omega)<\infty$.

For the proof of Theorem \ref{T1.2} we shall need an analytic inequality. Let $\omega$ be a nonempty compact subset of $\Omega_C$. For a continuous real function $f$ on $\omega$, we define
$$ \|f\|_L := \sup_{u,v\in\omega,\,u\not=v} \frac{|f(u)-f(v)|}{\|u-v\|},\qquad
\|f\|_\infty := \sup_{u\in\omega} |f(u)|,$$
$$ \|f\|_{BL} := \|f\|_L+\|f\|_\infty.$$
Note that $\|\cdot\|_{BL}$ depends on $\omega$, though this is not shown in the notation.

The following lemma extends a known estimate for probability measures to finite measures. The simple extension argument can be found in \cite[Sect. 3]{HS02}; or see \cite{Sch14}, proof of Theorem 8.5.3.

\begin{lemma}\label{L2.1} 
Let $\mu,\nu$ be finite Borel measues on $\omega$. Under the assumptions above, there is a constant $c_0$, depending only on the total measures $\mu(\omega),\nu(\omega)$ such that
$$ \left| \int_\omega f\,\D(\mu-\nu)\right| \le c_0\|f\|_{BL}\cdot \delta_{LP}(\mu,\nu).$$
\end{lemma}

\section{Proof of Theorem \ref{T1.1}}\label{sec3}

Let $\varphi$ be a finite Borel measure on $\Omega_C$. The idea for the proof of Theorem \ref{T1.1} (following an approach used in \cite{Sch18} for cone-volume measures) is to use Theorem A, where the given measure has compact support. Therefore, we choose a sequence $(\omega_j)_{j\in\N}$ of open subsets of $\Omega_C$ such that ${\rm cl }\,\omega_j\subset \omega_{j+1}$ (where ${\rm cl }$ denotes the closure) for all $j\in \N$ and $\bigcup _{j\in N} \omega_j=\Omega_C$. For each $j\in\N$, the measure $\varphi_j$ defined by $\varphi_j(A) := \varphi(A\cap\omega_j)$ for $A\in\B(\Omega_C)$ is defined on $\Omega_C$ and has compact support. Therefore, Theorem A can be applied to it. It yields a $C$-full set $L_j$ (uniquely determined according to Theorem B) such that $S_{d-1}(L_j,\cdot)=\varphi_j$. 

We must show that the sets $L_j$ do not `disappear to infinity' as $j\to\infty$. Let $j\in\N$ be given. Let $B$ be the largest ball with center $o$ such that $B\cap{\rm int}\,L_j=\emptyset$; let $r$ be its radius. Then there is a point $z\in B\cap L_j$. From (\ref{2.b}) and $S_{d-1}(L_j,\Omega_C) =\varphi_j(\Omega_C)\le \varphi(\Omega_C)$ it follows that $r\le t_0$ with a constant $t_0$ that depends only on $C$ and $\varphi(\Omega_C)$ (and not on $j$). If we choose $t_1>t_0$, then $z\in H^-(w,t_1)$. This shows that
\begin{equation}\label{3.1}
L_j\cap H^-(w,t_1) \not=\emptyset\quad\mbox{for all }j\in\N.
\end{equation}

Let $(t_k)_{k\in \N}$ be an increasing sequence of positive numbers with $t_k\uparrow \infty$ as $k\to\infty$. By (\ref{3.1}) we have $L_j\cap C_{t_1}\not=\emptyset$ for $j\in\N$, hence the bounded sequence $(L_j\cap C_{t_1})_{j\in\N}$ of convex bodies has a convergent subsequence. Thus, for a subsequence $(j_{1,i})_{i\in\N}$ of $\N$, there is a convex body $K_1$ satisfying
$$ L_{j_{1,i}}\cap C_{t_1}\to K_1\quad\mbox{as } i\to\infty.$$
For the same reason, there are a subsequence $(j_{2,i})_{i\in\N}$ of $(j_{1,i})_{i\in\N}$ and a convex body $K_2$ such that
$$ L_{j_{2,i}}\cap C_{t_2}\to K_2\quad\mbox{as } i\to\infty.$$
By induction, we obtain for each $k\in\N$ a subsequence $(j_{k,i})_{i\in\N}$ of $(j_{k-1,i})_{i\in\N}$ and a convex body $K_k$ such that
$$ L_{j_{k,i}}\cap C_{t_k}\to K_k\quad\mbox{as } i\to\infty.$$
The diagonal sequence $(\ell_i)_{i\in\N}:= (j_{i,i})_{i\in\N}$ then satisfies
$$ L_{\ell_i}\cap C_{t_k} \to K_k \quad\mbox{as }i\to\infty, \mbox{ for each }k\in\N.$$
For $1\le k<m$ we have
$$ L_{\ell_i}\cap C_{t_k}\to K_k,\quad  L_{\ell_i}\cap C_{t_m}\to K_m\quad\mbox{as }i\to\infty.$$
Using \cite[Thm. 1.8.10]{Sch14}, we obtain
\begin{eqnarray*} 
K_k &=& \lim_{i\to\infty}(L_{\ell_i}\cap C_{t_k})= \lim_{i\to\infty}[(L_{\ell_i}\cap C_{t_m})\cap C_{t_k}]\\
&=& \left[\lim_{i\to\infty}(L_{\ell_i}\cap C_{t_m})\right]\cap C_{t_k}=K_m\cap C_{t_k}.
\end{eqnarray*}
Therefore, if we define
$$ K:= \bigcup_{k\in\N} K_k,$$
then
$$ K\cap C_{t_k} = K_k\quad\mbox{for }k\in\N.$$
This implies, in particular, that $K\subset C$ is a closed convex set.

We have to show that
\begin{equation}\label{3.2}
S_{d-1}(K,\cdot)=\varphi.
\end{equation}
 
Let $j\in\N$, and let $\omega\subset\omega_{\ell_j}$ be an open set. Since $\omega$ is contained in a compact subset of $\Omega_C$, the set $\tau(K,\omega)$ is bounded, as shown in the previous section. Hence, there is a number $k\in\N$ with $\tau(K,\omega)=\tau(K_k,\omega)$. From $\lim_{i\to\infty}(L_{\ell_i}\cap C_{t_k})=K_k$ and the weak continuity of the surface area measure, it follows that
$$ S_{d-1}(K_k,\omega)\le \liminf_{i\to\infty} S_{d-1}(L_{\ell_i}\cap C_{t_k},\omega).$$
By the definition of $L_{\ell_i}$ we have $S_{d-1}(L_{\ell_i}\cap C_{t_k},\omega)=\varphi(\omega)$ for sufficiently large $i$, thus
\begin{equation}\label{8.7.3} 
S_{d-1}(K_k,\omega) \le\varphi(\omega).
\end{equation}
If $\beta\subset\omega_{\ell_j}$ is a closed set, then a similar argument yields that
\begin{equation}\label{8.7.4}
S_{d-1}(K_k,\beta)\ge\limsup_{i\to\infty} S_{d-1}(L_{\ell_i}\cap C_{t_k},\beta) =\varphi(\beta).
\end{equation}

Let $\beta\subset\omega_{\ell_j}$ be closed. We choose a sequence $(\eta_r)_{r\in\N}$ of open neighborhoods of $\beta$ with $\eta_r\subset\omega_{\ell_j}$ and $\eta_r\downarrow \beta$ as $r\to\infty$. By (\ref{8.7.3}), we have $S_{d-1}(K_k,\eta_r)\le\varphi(\eta_r)$. Since $\eta_r\downarrow\beta$, this gives $S_{d-1}(K_k,\beta)\le\varphi(\beta)$, and from (\ref{8.7.4}) we then conclude that $S_{d-1}(K_k,\beta)=\varphi(\beta)$.

For a closed set $\beta\subset\Omega_C$ with $\beta\subset\omega_{\ell_j}$ for some $j\in\N$, we have $\tau(K,\omega)=\tau(K_k,\omega)$ for suitable $k\in\N$, hence $S_{d-1}(K,\beta)= S_{d-1}(K_k,\beta)=\varphi(\beta)$. Since $\omega_{\ell_j}\uparrow\Omega_C$ as $j\to\infty$, the equality $S_{d-1}(K,\beta)=\varphi(\beta)$ holds for every closed set $\beta\in\B(\Omega_C)$ and thus for every Borel set in $\B(\Omega_C)$. We have proved the assertion (\ref{3.2}).

It remains to show that $K$ is $C$-close. For this, we note that by (\ref{2.2}) we have, for each $j\in\N$,
\begin{eqnarray*}
V_d(C\setminus L_j) &=& -\frac{1}{d} \int_{\Omega_C} h(L_j,u)\,S_{d-1}(L_j,\D u)\\
&=& -\frac{1}{d}\int_{\omega_j} h(L_j,u)\,\varphi(\D u).
\end{eqnarray*}
It follows from (\ref{2.a}) and (\ref{2.b}) that there is a constant $c_1$, depending only on $C$ and $K$, such that $|h(L_j,\cdot)| \le c_1$. This yields
$$ V_d(C\setminus L_j) \le \frac{c_1}{d}\varphi(\Omega_C).$$
For $i,k\in\N$ we get
$$ V_d(C_{t_k}\setminus L_{\ell_i})\le V_d(C\setminus L_{\ell_i})\le \frac{c_1}{d}\varphi(\Omega_C).$$
Since $L_{\ell_i}\cap C_{t_k}\to K_k$ as $i\to\infty$, this gives
$$ V_d(C_{t_k}\setminus K)= V_d(C_{t_k}\setminus K_k)\le \frac{c_1}{d}\varphi(\Omega_C).$$
This holds for all $k\in\N$, hence we deduce that
$$ V_d(C\setminus K) \le \frac{c_1}{d}\varphi(\Omega_C)<\infty.$$
Thus, $K$ is a $C$-close set. \hfill$\Box$

\section{Examples, and a necessary condition}\label{sec4}

Our first example shows that the $C$-close set $K$ that exists by Theorem \ref{T1.1} need not necessarily be $C$-full; in other words, $K$ can have finite total surface area measure and nevertheless $C\setminus K$ can be unbounded.

To provide an example, we choose $d=3$ and let $C$ be the positive orthant, that is, the positive hull of the standard orthonormal basis of $\R^3$. With respect to this basis, we define points $p_n,q_n$ by their coordinate triples,
$$ p_n=(a_n,0,n-1),\quad q_n=(0,a_n,n-1) \quad\mbox{with } a_n:= \frac{1}{n^2}\quad\mbox{for }n\in\N.$$
Then we let $K$ be the convex hull of the points $p_n,q_n$, $n\in\N$, and the rays $\{(x,0,0):x\ge 1\}$, $\{(0,x,0):x\ge 1\}$. It is easy to check that all points $p_n,q_n$ are vertices of $K$ and that
$$ S_{2}(K,\Omega_C)= \frac{1}{\sqrt{2}} \sum_{n=1}^\infty (a_n+a_{n+1})<\infty.$$
Clearly, $C\setminus K$ is unbounded.

In our second example, $d=2$, and $C$ is the positive orthant in $\R^2$. We define $K$ as the convex hull of the curves
$$ \{(x,x^2): x\le 0\},\quad \{(x,-\sqrt{x}): x\ge 0\}.$$
Clearly, $S_1(K,\omega)<\infty$ if $\omega$ is a compact subset of $\Omega_C$. Every closed convex set with spherical image $\Omega_C$ and surface area measure equal to $S_1(K,\cdot)$ is a translate of $K$ (as can be deduced from \cite[Thm. 8.3.3]{Sch14}). However, no translate of $K$ is contained in $C$. This shows (for $d=2$) that the local finiteness (that is, finiteness on compact subsets of $\Omega_C$) of a Borel measure on $\Omega_C$ is not sufficient to be the surface area measure of a $C$-asymptotic set. The following assertion shows this also for higher dimensions.

For a compact set $\omega\subset\Omega_C$, we denote by $\Delta(\omega)$ the distance of $\omega$ from ${\rm bd}\,\Omega_C$ (the boundary of $\Omega_C$ with respect to $\Sd$),  that is, the smallest angle between a vector $u\in\omega$ and a vector $v\in{\rm bd}\,\Omega_C$.

\begin{theorem}\label{T4.1} 
Let $\varphi$ be a Borel measure on $\Omega_C$. If $\varphi$ is the surface area measure of some $C$-asymptotic convex set $K$, then the function
$$ 
\omega\mapsto\Delta(\omega) ^{d-1}\varphi(\omega),\quad \omega\subset\Omega_C \mbox{ compact},
$$
is bounded.
\end{theorem}

\begin{proof}
Supppose  $K$ is a $C$-asymptotic convex set with $S_{d-1}(K,\cdot)=\varphi$. Let $\omega\subset\Omega_C$ be a nonempty compact set. 

Let $B$ be the largest ball with center $o$ not meeting ${\rm int}\,K$, let $r$ be its radius. Let $x\in\tau(K,\omega)$, and let $H(u,s)$ be a supporting hyperplane of $K$ at $x$ with outer normal vector $u\in\omega$. This hyperplane must meet $B$, hence its distance $s$ from $o$ satisfies 
$$s\le r.$$ 
Let $y\in H(u,s)\cap C$ be a point for which $\langle w,y\rangle$ is maximal, then $y\in{\rm bd C}$. Let $t=\langle w,y\rangle$, then $y\in H(w,t)$ and $C\cap H(u,s)\subset H^-(w,t)$. Let
$$ E:= H(u,s)\cap H(w,t).$$
The $(d-2)$-plane $E$ is a supporting plane of $C\cap H(w,t)$ in $H(w,t)$. Therefore, there is a supporting hyperplane $H(v,0)$ of the cone $C$ (with outer normal vector $v$) such that $H(v,0)\cap H(w,t)=E$. Let $\alpha$ be the angle between $u$ and $v$. Since $u\in\omega$ and $v\in{\rm bd}\,\Omega_C$, we have 
$$\alpha\ge\Delta(\omega).$$
Let $z$ be the image of $o$ under orthogonal projection to $H(u,s)$, and let $p$ be the image of $z$ under orthogonal projection to $E$. Let $\gamma$ be the angle of the triangle with vertices $o,z,p$ at $p$ Since the vector $z-p$ is orthogonal to $E$, we have 
$$\gamma\ge \alpha.$$ 
To see this, note that for $d=2$ we have $p=y$ and $\gamma=\alpha$. Let $d\ge 3$. The two-dimensional plane through $p$ orthogonal to $E$ contains $z$. We choose $q\in H(v,0)\setminus\{p\}$ such that $p$ is the orthogonal projection of $q$ to $E$. Then the unit vectors $-p/\|p\|, (z-p)/\|z-p\|, (q-p)/\|q-p\|$ are the vertices of a spherical triangle (on a two-dimensional unit sphere) of side lengths $\gamma$ (opposite to a right angle), $\alpha$, and some $\beta$. Then $\cos\gamma=\cos\alpha\cos\beta\le\cos\alpha$, hence $\gamma\ge \alpha$.  It follows that
$$ \sin\Delta(\omega)\le\sin\alpha\le\sin\gamma=\frac{s}{\|p\|}$$
and hence
$$\|p\|\sin\Delta(\omega)\le r.$$
Moreover, 
$$ \langle w,y\rangle=\langle w,p\rangle \le \|p\|\le \frac{r}{\sin\Delta(\omega)}.$$
Since $t=\langle w,y\rangle$, we have
$$ \tau(K,\omega) \subset H^-(w,t)\quad\mbox{and}\quad t\le\frac{c_3}{\sin\Delta(\omega)},$$
with a constant $c_3$ depending on $C,w,K$, but not on $\omega$.

Considering the orthogonal projection of $\tau(K,\omega)$ to the hyperplane $H(w,t)$ and denoting by $V_{d-1}$ the $(d-1)$-dimensional volume, we obtain
$$ \int_\omega |\langle w,u\rangle|\,S_{d-1}(K,\D u) \le V_{d-1}(C\cap H(w,t)) = t^{d-1}V_{d-1}(C\cap H(w,1))\le \frac{c_4}{\sin^{d-1}\Delta(\omega)},$$
with a constant $c_4$ that again depends on $C,w,K$, but not on $\omega$. 

We recall that the unit vector $w$ was chosen such that $-w\in{\rm int}\,C^\circ$. We can choose $w$ such that in addition $w\in {\rm int}\,C$. (In fact, if ${\rm int}\,C\cap{\rm int}(-C^\circ)=\emptyset$, then $C$ and $-C^\circ$ can be separated by a hyperplane, by \cite[Thm. 1.3.8]{Sch14}, hence $C$ and $C^\circ$ lie in the same closed halfspace, a contradiction.) With this choice, we have $|\langle w,u\rangle|\ge c_5$ for all $u\in\Omega_C$, with a constant $c_5$ depending only on $C$ and $w$, and hence $\varphi(\omega) \le c_6\sin^{1-d}\Delta(\omega)$
for each compact subset $\omega\subset\Omega_C$, with a constant $c_6$ that is independent of $\omega$. Further, $\Delta (\omega)\le \pi/2$ and hence $\sin\Delta(\omega)\ge(2/\pi)\Delta(\omega)$. This completes the proof.
\end{proof}

It is easy to construct a measure $\varphi$ on $\Omega_C$ which is finite on compact subsets of $\Omega_C$ but does not satisfy the condition of Theorem \ref{T4.1}. This shows that the statement on pp. 344--345 (``Generalization of the Theorem of Section 7.4'') in Alexandrov \cite{Ale05} is not correct.

\section{Proof of Theorem \ref{T1.2}}\label{sec5}

Let $\omega\subset\Omega_C$ be a nonempty compact set. This set is fixed in the present section, and the norm $\|\cdot\|_{BL}$ introduced in Section \ref{sec2} depends on it. 

\begin{lemma}\label{L5.1}
Let $K$ be a $C$-full set. There exists a constant $c_7$, depending only on $C$, $\omega$ and an upper bound for $S_{d-1}(K,\omega)$, such that
$$ \| h(K,\cdot)\|_{BL} \le c_7.$$
\end{lemma}

\begin{proof} 
By (\ref{2.a}) and (\ref{2.b}) we have
\begin{equation}\label{H1}
|h(K,\cdot)|\le c_1,
\end{equation}
where the constant $c_1$ depends only on $C$, $\omega$ and an upper bound for $S_{d-1}(K,\omega)$.

We need a Lipschitz constant for the function $|h(K,\cdot)|$. Since $\omega$ is a compact subset of $\Omega_C$, there is by (\ref{2.c}) a positive constant $a$, depending only on $C$ and $\omega$, such that
$$ |\langle x,u\rangle|\ge a\|x\|\quad \mbox{for }x\in C \mbox{ and }u\in\omega.$$
Let $x\in K$ and $u\in\omega$ be such that $\langle x,u\rangle=h(K,u)$. Then we obtain
\begin{equation}\label{H1a} 
\|x\|\le \frac{1}{a} |\langle x,u\rangle| = \frac{1}{a}|h(K,u)| \le \frac{c_1}{a}=:c_8.
\end{equation}
Also the constant $c_8$ depends only on $C$, $\omega$ and an upper bound for $S_{d-1}(K,\omega)$

Now let $u,v\in\omega$ and choose $x\in K$ with $h(K,u)=\langle x,u\rangle$. We have $\langle x,v\rangle \le h(K,v)$ and hence 
$$h(K,u)-h(K,v)\le \langle x,u-v\rangle\le\|x\|\|u-v\| \le c_8\|u-v\|.$$
Here $u$ and $v$ can be interchanged. Therefore
\begin{equation}\label{H2} 
\|h(K,\cdot)\|_L \le c_8.
\end{equation}
From (\ref{H1}) and (\ref{H2}) the assertion of the lemma follows.
\end{proof}

\noindent{\em Proof of Theorem} \ref{T1.2}.

Let $K,L$ be $C$-full sets whose surface area measures are concentrated on the given compact set $\omega$. Let $x\in{\rm bd}\,K\cap{\rm int}\, C$. There is an outer normal vector $u\in\omega$ to  $K$ at $x$, hence (\ref{H1a}) shows that $\|x\|\le c_8$, where $c_8$ depends only on $C,\omega$ and an upper bound for $S_{d-1}(K,\omega)$. Thus, ${\rm bd}\,K\cap{\rm int}\,C$, and similarly ${\rm bd}\,L\cap{\rm int}\,C$, is contained in the ball with center $o$ and radius $c_8$ (with $c_8$ increased with respect to $L$, if necessary). We choose $t>0$ so large that
$$ {\rm bd}\,K\cap{\rm int}\,C \subset H^-(w,t),\quad {\rm bd}\,L\cap{\rm int}\,C \subset H^-(w,t)$$
and that
$$ K_t:= K\cap H^-(w,t),\quad L_t:= L\cap H^-(w,t)$$
have inradius at least $1$. How large $t$ has to be chosen to achieve this effect, depends only on $C,\omega$ and an upper bound for $S_{d-1}(K,\omega), S_{d-1}(L,\omega)$. Then there is a number $R$, also depending only on these data, such that $K_t,L_t$ have circumradius at most $R$.

Now we set  $\delta_{LP}(S_{d-1}(K,\cdot), S_{d-1}(L,\cdot))=:\varepsilon$. By Lemmas \ref{L2.1} and \ref{L5.1} we have
$$ \left| \int_\omega h(K,\cdot)\,\D\left(S_{d-1}(K,\cdot)-S_{d-1}(L,\cdot)\right)\right| \le c_0\| h(K,\cdot)\|_{BL}\cdot\varepsilon \le c_0c_8\varepsilon.$$
We have a disjoint decomposition
$$ \Sd = \omega\cup(\Sd\cap {\rm bd}\,C^\circ) \cup \{w\} \cup N$$
with $N:= \Sd\setminus(\omega\cup {\rm bd}\,C^\circ \cup \{w\})$. 
Here,
\begin{eqnarray*} 
h(K_t,u) &=& h(K,u)\quad\mbox{for }u\in\omega,\\
S_{d-1}(K_t,A) &=& S_{d-1}(K,A) \quad\mbox{for } A\in\B(\omega),\\
h(K_t,u) &=& 0\quad\mbox{for }u\in{\rm bd}\,C^\circ,
\end{eqnarray*}
and similarly for $L,L_t$. Further,
$$ S_{d-1}(K_t,\{w\})= V_{d-1}(C\cap H(w,t))=S_{d-1}(L_t,\{w\}),$$
$$ S_{d-1}(K_t,A)= 0 = S_{d-1}(L_t,A)\quad\mbox{for }A\in\B(N),$$
the latter because a boundary point of $K_t$ or $L_t$ with outer normal vector in $N$ is a singular point. Therefore, for the volume and mixed volumes of the convex bodies $K_t$ and $L_t$ we obtain from \cite[(5.19)]{Sch14} that
\begin{eqnarray*}
&& d\left|V_d(K_t)-V(K_t,L_t,\dots,L_t)\right|\\
&& = \left| \int_{\Sd} h(K_t,u)\,\left(S_{d-1}(K_t,\D u)-S_{d-1}(L_t,\D u)\right)\right|\\
&& =\left| \int_\omega h(K,u)\left(S_{d-1}(K,\D u)-S_{d-1}(L,\D u)\right)\right|.
\end{eqnarray*}
We conclude that
\begin{equation}\label{8.4.1}
\left|V_d(K_t)-V(K_t,L_t,\dots,L_t)\right| \le c_9\varepsilon,
\end{equation}
and since $K$ and $L$ may be interchanged, also
\begin{equation}\label{8.4.2}
\left|V_d(L_t)-V(L_t,K_t,\dots,K_t)\right| \le c_9\varepsilon,
\end{equation}
with $c_9=c_0c_8/d$.

As the proof of Theorem 8.5.1 and Lemma 8.5.2 in \cite{Sch14} shows, the two inequalities (\ref{8.4.1}) and (\ref{8.4.2}), together with $\varepsilon\le \varepsilon_0$ for a suitable $\varepsilon_0>0$, are sufficient to obtain an inequality of the form
$$ d_H(K_t,L_t)\le \gamma(c_7\varepsilon)^{1/d}$$
with a constant $\gamma$ that depends only on the dimension $d$, a positive lower bound for the inradius of $K_t,L_t$ and an upper bound for the circumradius of $K_t,L_t$. The restriction $\varepsilon\le \varepsilon_0$ can later be removed by adapting the constant, as remarked in \cite[p. 33]{HS02}. This yields the assertion.  \hspace*{\fill}$\Box$

\noindent Author's address:\\[2mm]
Rolf Schneider\\Mathematisches Institut, Albert-Ludwigs-Universit{\"a}t\\D-79104 Freiburg i.~Br., Germany\\E-mail: rolf.schneider@math.uni-freiburg.de

\end{document}